\documentclass[11pt,reqno]{amsart}
\usepackage{amssymb,microtype}
\usepackage[hidelinks]{hyperref}
\hypersetup{pdftitle={The mixing exponent of a heavy-tailed random intersection graph},
  pdfauthor={Vyacheslav Koval},pdfkeywords={random intersection graph, mixing time, feature clique, spectral gap}}
\calclayout
\numberwithin{equation}{section}
\newtheorem{theorem}{Theorem}[section]
\newtheorem{proposition}[theorem]{Proposition}
\newtheorem{lemma}[theorem]{Lemma}
\newcommand{\PP}{\mathbb P}
\newcommand{\EE}{\mathbb E}
\newcommand{\cC}{\mathcal C_1}
\newcommand{\tmix}{t_{\mathrm{mix}}}
\newcommand{\trel}{t_{\mathrm{rel}}}
\newcommand{\probto}{\xrightarrow{\PP}}
\newcommand{\law}{\xrightarrow{\mathrm d}}
\DeclareMathOperator{\vol}{vol}
\DeclareMathOperator{\Geom}{Geom}
\DeclareMathOperator{\Unif}{Unif}

\title[Mixing on a heavy-tailed random intersection graph]{The mixing exponent of a heavy-tailed\protect\\ random intersection graph}
\author{Vyacheslav Koval}
\address{Bernoulli Institute, University of Groningen, The Netherlands}
\email{v.v.koval@rug.nl}
\date{}
\begin{document}
\begin{abstract}
We determine the polynomial order of the worst-case mixing time on the giant component of a sparse random intersection graph. If the feature weights have a regularly varying tail of index $-\beta$, with $\beta>2$, and the model is supercritical, both the mixing time and the relaxation time are $n^{1/\beta+o_{\mathbb P}(1)}$.
The lower bound comes from the largest feature cliques. A clique of size $s$ at the extreme scale has asymptotically $s\,\mathbb E W^2$ outgoing edges, and its first exit time divided by $s$ is asymptotically exponential with rate $\mathbb E W^2/2$, uniformly over starting vertices. For the upper bound we use size-two features to construct a sparse spanning graph with a polylogarithmic Poincar\'e constant. We also bound the number of small eigenvalues and determine the first-order length of the longest pendant path.
\end{abstract}
\maketitle

\section{Introduction}
A feature shared by $s$ people creates $\binom{s}{2}$ edges in the graph of their pairwise affiliations. If the group has order $s$ edges to the rest of the graph, only a fraction of order $1/s$ of its incident edge ends point outwards. This suggests an escape time of order $s$. The mixing problem is to decide whether the walk can take much longer to explore the giant component once it can leave any one group. We show that the largest group determines the exponent of the mixing time.

Fix $c>0$. There are $n$ vertices and $m=\lfloor cn\rfloor$ features. The feature weights $W_1,\ldots,W_m$ are independent copies of a nonnegative random variable $W$. Conditional on the weights, all vertex--feature incidences are independent, with
\begin{equation}\label{eq:model}
 p_j=\min\{1,W_j/(n\sqrt c)\}.
\end{equation}
Let $F_j$ be the set of vertices carrying feature $j$, and write $K_j=|F_j|$. Two distinct vertices of $G_n$ are adjacent if they share at least one feature. An edge carries unit weight even when several features generate it. Our assumptions are
\begin{equation}\label{eq:hyp}
 \bar F(x):=\PP(W>x)=x^{-\beta}\ell(x),\qquad
 \beta>2,\qquad \nu:=\EE W^2>1,
\end{equation}
where $\ell$ is slowly varying at infinity. The exponent assumption gives $\nu<\infty$. Indeed, the expected number of edges, counted with feature multiplicity, is $m\binom{n}{2}\EE p_1^2\sim \nu n/2$. A feature entered from a typical vertex has a size-biased weight, and the expected number of vertices reached in two steps of the limiting incidence exploration is $\nu$. The condition $\nu>1$ is therefore the supercriticality condition.

Let $\cC$ be a largest component of $G_n$, breaking ties by a fixed rule. On this component consider the discrete-time walk with transition matrix $P=(I+D^{-1}A)/2$, where $D$ and $A$ are its degree and adjacency matrices. Write $d(v)$ for its degrees, $\vol(S)=\sum_{v\in S}d(v)$, and $\pi(v)=d(v)/\vol(\cC)$. For $0<\varepsilon<1$ set
\[
 \tmix(\varepsilon)=\min\bigl\{t\ge0:\max_{v\in\cC}\|P^t(v,\cdot)-\pi\|_{\rm TV}\le\varepsilon\bigr\},
 \qquad \trel=1/\lambda_2(I-P).
\]
All assertions about the walk concern the event $|\cC|\ge2$, which has probability tending to one. Write $K_n^*=\max_{j\le m}K_j$ and choose $a_n$ so that $m\bar F(\sqrt c\,a_n)\to1$.

\begin{theorem}[Mixing and relaxation]\label{thm:main}
Assume \eqref{eq:hyp}. For every fixed $0<\varepsilon<1$,
\begin{equation}\label{eq:exponents}
 \frac{\log\tmix(\varepsilon)}{\log n}\probto\frac1\beta,\qquad
 \frac{\log\trel}{\log n}\probto\frac1\beta.
\end{equation}
More precisely, put $b_\varepsilon=2\log(1/\varepsilon)/\nu$. There are finite constants $C$ and $C_\varepsilon$ such that for every fixed $\delta>0$, with probability tending to one,
\begin{align}
 (b_\varepsilon-\delta)K_n^*&\le\tmix(\varepsilon)\le C_\varepsilon K_n^*(\log n)^8,\label{eq:mixbounds}\\
 (2/\nu-\delta)K_n^*&\le\trel\le C K_n^*(\log n)^7.\label{eq:relbounds}
\end{align}
Moreover, $K_n^*/a_n\law Z_\beta$, where $\PP(Z_\beta\le z)=\exp(-z^{-\beta})$ for $z>0$.
\end{theorem}

For the Pareto law $\PP(W>x)=x^{-\beta}$, $x\ge1$, one can take $a_n=m^{1/\beta}/\sqrt c$ and $\nu=\beta/(\beta-2)$. For example, $\beta=5/2$ gives mixing exponent $2/5$. The random normalization in \eqref{eq:mixbounds} retains the fluctuations of the largest feature. The logarithmic powers on the right arise from the comparison argument. A limiting law for $\tmix(\varepsilon)/a_n$ remains open.

\subsection*{The local time scale}
In decreasing order of weight, let the feature indices be $J_1,J_2,\ldots$, with ties broken by index. For a feature clique $S$ in the giant, let $\tau_S=\inf\{t\ge1:X_t\notin S\}$, let $E_S=(X_{\tau_S-1},X_{\tau_S})$ and let $\partial S$ denote the edges of the cut oriented out of $S$. The law $\mathcal L_x$ of the walk conditions on the realized vertex--feature incidences and starts at $x$. We use the geometric distribution on $\{1,2,\ldots\}$.

\begin{theorem}[Exit from an extreme feature]\label{thm:exit}
For every fixed $r$, with probability tending to one each $S=F_{J_i}$, $1\le i\le r$, lies in the giant and has a nonempty boundary. For these cliques, simultaneously,
\begin{equation}\label{eq:edgeexit}
 \max_{x\in S}\left\|
 \mathcal L_x(\tau_S,E_S)-
 \Geom\!\left(\frac{\nu}{2|S|}\right)\otimes\Unif(\partial S)
 \right\|_{\rm TV}\probto0 .
\end{equation}
In particular, uniformly in $x\in S$ and $t\ge0$,
$\PP_x(\tau_S>t|S|)=e^{-\nu t/2}+o_{\mathbb P}(1)$.
\end{theorem}

Averaging over the graph gives a second limit. The law of $\tau_{F_{J_1}}/a_n$ converges to $(2/\nu)Z_\beta E$, where $E$ is a mean-one exponential independent of $Z_\beta$. The start can be any vertex chosen from the clique using the marked graph. Section~\ref{sec:spectrum} counts the low-lying eigenvalues produced by many such cliques. Section~\ref{sec:paths} gives a different geometric statistic, the exact first-order length of the longest pendant path in the simple projection.

\subsection*{Relation to previous work}
Bloznelis and Damarackas~\cite{BD} derived the typical degree law for weighted random intersection graphs with weights on both parts of the incidence graph. Nikoletseas, Raptopoulos and Spirakis~\cite{NRS} proved logarithmic mixing for homogeneous intersection graphs in a regime near connectivity. In the supercritical Erd\H{o}s--R\'enyi graph with fixed mean degree, the worst-case mixing time is of order $(\log n)^2$~\cite{BKW,FR}. Pendant paths of logarithmic length occur in our model as well, but extreme feature cliques give a polynomial time scale. Nannini and Zanette~\cite{NZ} studied residence times in a model of disjoint cliques with external links. In the present model the cliques can overlap, and their external degrees are random and unbounded.

The upper bound uses the expansion result of Krivelevich, Reichman and Samotij~\cite{KRS} for a connected graph with added uniform random edges. Here the graph exposed before these edges is disconnected. We partition the components outside one linear component into connected sets and use size-two features to attach those sets. This step gives a spanning comparison on the entire giant.

\section{Extreme features and their boundaries}\label{sec:boundary}

We first record the geometric estimates used in the proofs. A bounded part of the incidence graph already supplies a linear component. Choose $0<q<1$ and $B<\infty$, with $B$ a continuity point of the weight law, such that
\begin{equation}\label{eq:trunc}
 q\,\EE[W^2\mathbf1_{\{W\le B\}}]>1.
\end{equation}
In the first $\lfloor qm\rfloor$ columns retain only weights at most $B$. The bounded-kernel giant-component theorem of Bollob\'as, Janson and Riordan~\cite[Theorems 3.1 and 9.10]{BJR} gives a component containing at least $\rho n$ people with high probability, for some $\rho>0$. To check its normalization, the vertex types are people and feature weights. From a person the limiting intensity of retained feature weights is $q\sqrt c\,w\mathbf1_{\{w\le B\}}\,dF(w)$, and from a feature of weight $w$ the mean number of people is $w/\sqrt c$. The two-step mean is the left-hand side of \eqref{eq:trunc}. On the positive-weight retained types this bounded bipartite kernel is irreducible, with positive survival probability at the person type. Zero-weight and discarded feature types are isolated. The convergence of the empirical type measures follows from the law of large numbers, and boundedness gives the required convergence of edge densities. Removing one specified column leaves the same limit. Proposition~\ref{prop:upper} below also shows that the full graph has a unique linear component.

For reference, let $Y_i$ be the sum of $i$ independent mean-one exponentials, taken as partial sums of a single independent sequence. Regular variation and binomial rare-event limits give, for every fixed $r$,
\begin{equation}\label{eq:extremes}
 \left(\frac{W_{J_i}}{\sqrt c\,a_n}\right)_{i\le r}\law
       (Y_i^{-1/\beta})_{i\le r},\qquad a_n=n^{1/\beta+o(1)} .
\end{equation}
Indeed, the counts in disjoint scaled intervals $(u,v]$ converge jointly to independent Poisson variables of means $u^{-\beta}-v^{-\beta}$, by the multinomial formula. Inverting the ordered points gives \eqref{eq:extremes}. Conditional on the weights, $K_j\sim\mathrm{Bin}(n,p_j)$. On $\max_j W_j/\sqrt c\le A a_n$, the probability of a deviation $|K_j-W_j/\sqrt c|>\delta a_n$ anywhere is at most $2m e^{-b_{A,\delta}a_n}$ for large $n$, by a binomial tail bound. Letting $A\to\infty$ shows
\begin{equation}\label{eq:maxfeature}
 \max_{j\le m}|K_j-W_j/\sqrt c|=o_{\mathbb P}(a_n),\qquad
 K_{J_1}/K_n^*\probto1.
\end{equation}

For $S=F_j$, let $b_x$ be the number of neighbors of $x\in S$ outside $S$ and $b_j=\sum_{x\in S}b_x$. The following estimate requires only a finite second moment for $W$.

\begin{lemma}[Boundary of a tagged clique]\label{lem:boundary}
Condition on the weight and the membership set of a specified feature $j$. If $s=|F_j|\to\infty$ and $s=o(n)$, then
\begin{equation}\label{eq:boundary}
  b_j/s\probto\nu,\qquad s^{-1}\max_{x\in F_j} b_x\probto0.
\end{equation}
The estimates are uniform over $s$ in any range $s_-\le s\le s_+$ with $s_-\to\infty$ and $s_+/n\to0$.
\end{lemma}

\begin{proof}
Given $S=F_j$, the other columns retain their original independent laws. For $h\ne j$, write $U_h=|F_h\cap S|$ and $V_h=|F_h\setminus S|$. Given $W_h$, these are independent binomials with parameters $(s,p_h)$ and $(n-s,p_h)$. The number of boundary edges counted with their feature multiplicities is $Z=\sum_{h\ne j}U_hV_h$. If $p$ has the law of a single membership probability, then
\[
 \EE[Z]/s=(m-1)(n-s)\EE p^2\longrightarrow\nu .
\]
To prove concentration, fix $B<\infty$ and keep the terms with $W_h\le B$. Their sum $Z_B$ has variance at most $C_B(s+s^2/n)$, by independence of the columns and the binomial second-moment bound $\EE[U_h^2\mid W_h]\le sp_h+(sp_h)^2$, with the analogous bound for $V_h$. Moreover $\EE[Z-Z_B]/s\le\EE[W^2\mathbf1_{\{W>B\}}]$. First let $n$ and then $B$ tend to infinity. For a pair $(x,y)\in S\times S^c$, each other feature witnesses the edge with probability $\EE p^2$. Counting pairs of witnesses thus gives
\[
  \frac{\EE[Z-b_j]}s\le (n-s)\binom{m-1}{2}(\EE p^2)^2=O(1/n).
\]
For the maximum, let $D_{x,B}=\sum_{h\ne j}\mathbf1_{\{W_h\le B,\ x\in F_h\}}V_h$. Its second moment is bounded by $C_B$ uniformly in $x$ and $n$, since each summand has first and second moments $O_B(1/n)$ and the columns are independent. The omitted part of every $b_x$ is bounded by $Z-Z_B$. Consequently, for $\delta>0$,
\[
 \PP\!\left(\max_{x\in S} b_x>\delta s\right)
 \le \frac{4C_B}{\delta^2s}+\frac2\delta\EE[W^2\mathbf1_{\{W>B\}}].
\]
These bounds give the stated uniformity.
\end{proof}

The lemma also applies at a fixed number of extreme ranks. For $0<l<L<\infty$ consider all features with $W_j/(\sqrt c\,a_n)\in(l,L]$. Their expected number stays bounded. Conditional on any weight in this band, the size is between $la_n/2$ and $2La_n$ with probability tending uniformly to one. Summing the tagged failure probabilities proves \eqref{eq:boundary} simultaneously throughout the band. By \eqref{eq:extremes} one can capture the first $r$ ranks with arbitrarily high probability this way. Their sizes exceed every fixed multiple of $\log n$ with high probability, so Proposition~\ref{prop:upper} puts them in $\cC$. Since $\vol(\cC)\ge 2(|\cC|-1)$, we also have, for each such $S$,
\begin{equation}\label{eq:smallmass}
 \vol(S)=s(s-1)+b_j,\qquad
 \pi(S)=O_{\mathbb P}(a_n^2/n)=o_{\mathbb P}(1).
\end{equation}

\section{Escape from a clique}\label{sec:exit}

The following calculation takes place on a fixed finite graph. It uses the
complete graph inside the set and allows different escape probabilities
at every vertex.

\begin{lemma}\label{lem:detexit}
Let $S$ be a clique of size $s$ in a finite simple graph. Write $b_x$
for the number of neighbors of $x\in S$ outside $S$. Along any sequence
with $s\to\infty$, suppose
\[
   \frac{1}{s}\sum_{x\in S}b_x\longrightarrow b\in(0,\infty),
   \qquad \max_{x\in S}b_x=o(s).
\]
For the half-lazy walk, \eqref{eq:edgeexit} holds with $b$ in place
of $\nu$ and ordinary convergence to zero.
\end{lemma}

\begin{proof}
The transition matrix $K$ of the walk killed on exiting $S$ has diagonal
$1/2$ and entries $1/(2d_x)$ from $x$ to $y\ne x$ in $S$, where
$d_x=s-1+b_x$. Write $\rho$ for its Perron eigenvalue and $h>0$ for
a right eigenvector. Its row sums lie in
$[1-\max_x b_x/(2d_x),1]$, so $\rho\to1$. Since the boundary is
nonempty for large $s$, irreducibility gives $\rho<1$. Set
$H=\sum_{y\in S}h_y$. The eigenvector equation gives exactly
\[
   h_x=\frac{H}{1+(2\rho-1)d_x}.
\]
Thus $\max_x h_x/\min_x h_x\to1$. Multiplying the eigenvector equation
by $d_x$ and summing over $S$ yields
\begin{equation}\label{eq:killedrate}
  (1-\rho)\sum_{x\in S}d_x h_x=\tfrac12\sum_{x\in S}b_xh_x,
  \qquad 1-\rho\sim \frac b{2s}.
\end{equation}

By reversibility of $K$ with weights $d_x$, the distribution
$\xi(x)=d_xh_x/\sum_y d_yh_y$ is a left eigenvector with eigenvalue
$\rho$. From this initial law the walk, conditional on survival, has
law $\xi$ after each step. It follows that $\tau_S$ is geometric with
parameter $1-\rho$ and independent of $E_S$. The latter law assigns to
each outward edge $(x,y)$ probability $h_x/\sum_{z\in S}b_zh_z$,
which differs from the uniform edge probability by a relative $o(1)$
uniformly over the cut.

It remains to compare an arbitrary start to $\xi$. Each row of $K$
dominates $\gamma$ times the uniform distribution on $S$, for
$\gamma=s/(2\max_x d_x)\to1/2$. Couple the two walks by using this
common uniform step with probability $\gamma$ until they meet or one
exits. At each step before coupling the probability that either exits
is at most $2e_*$, where $e_*=\max_x b_x/(2d_x)=o(1)$. The chance of
an exit before the common step is at most $2e_*/\gamma$. Once the walks
agree, use the same moves, including the eventual exit edge. This bounds
the total variation distance of the two joint laws, uniformly in the
initial point. Finally, two geometric distributions with parameters
$p,p'$ have total variation distance at most $|p-p'|/\max(p,p')$,
by coupling their Bernoulli trials. Apply \eqref{eq:killedrate}.
\end{proof}

\begin{proof}[Proof of Theorem~\ref{thm:exit}]
Apply Lemma~\ref{lem:detexit} to the boundary estimates
\eqref{eq:boundary} at the first $r$ ranks. Convergence in probability
of the graph estimates suffices, as every subsequence has a further
subsequence on which they hold almost surely. The geometric survival
function at $\lfloor st\rfloor$ converges to $e^{-\nu t/2}$ uniformly
for $t\ge0$, by convergence on compact intervals and monotonicity
in the tail. The annealed product law follows by conditioning on the
graph and using \eqref{eq:extremes}--\eqref{eq:maxfeature}.
\end{proof}

\begin{proof}[Proof of Theorem~\ref{thm:main}]
Take $S=F_{J_1}$ and $s=|S|$. For every $x\in S$,
\begin{equation}\label{eq:tvlower}
 \|P^t(x,\cdot)-\pi\|_{\rm TV}
 \ge P^t(x,S)-\pi(S)\ge\PP_x(\tau_S>t)-\pi(S).
\end{equation}
For any $0<u<b_\varepsilon$, the last expression at $t=\lfloor us\rfloor$
exceeds $\varepsilon$ with high probability by
Theorem~\ref{thm:exit} and \eqref{eq:smallmass}. For relaxation, test the
variational characterization of $\lambda_2(I-P)$ with $\mathbf1_S$.
The quotient is $b_{J_1}/[2\vol(S)(1-\pi(S))]$, which is
$(\nu+o_{\mathbb P}(1))/(2s)$. These inequalities and $s/K_n^*\probto1$
give the two lower bounds. Proposition~\ref{prop:upper} gives the
upper bounds. The Fr\'echet limit and the exponent follow from
\eqref{eq:extremes}--\eqref{eq:maxfeature}.
\end{proof}

\section{A sparse spanning graph}\label{sec:upper}

\begin{proposition}\label{prop:upper}
With high probability every component other than $\cC$ has $O(\log n)$
vertices and
\[
  \trel\le C K_n^*(\log n)^7,\qquad
  \tmix(\varepsilon)\le C_\varepsilon K_n^*(\log n)^8
  \quad (0<\varepsilon<1).
\]
Here $\varepsilon$ is fixed.
\end{proposition}

The comparison graph keeps every vertex and replaces each large clique by a
path. Features containing exactly two vertices supply independent random
edges. We first record how these edges connect the components of the
comparison graph. Write
$\mathcal E_J(f)=\sum_{uv\in E(J)}(f(u)-f(v))^2$, with each edge counted once.

\begin{lemma}\label{lem:sprinkle}
Let $H$ be a deterministic graph on $n$ vertices, of maximum degree at most
$b\log n$, with a component $C$ of size at least $an$. Add $M$ independent
uniform unordered pairs of distinct vertices, where $an\le M\le bn$ and
$a,b>0$ are fixed. Suppress repeated edges. With high probability the
component $J$ containing $C$ is the unique largest component, every other
component has $O(\log n)$ vertices, and for every $f:V(J)\to\mathbb R$,
\begin{equation}\label{eq:sparse-poincare}
  \min_{z\in\mathbb R}\sum_{v\in V(J)}(f(v)-z)^2
       \le C_0(\log n)^6\,\mathcal E_J(f).
\end{equation}
The probability estimates and the implicit constants are uniform over $H$ and $M$ satisfying these hypotheses.
\end{lemma}

\begin{proof}
Put $D=V(H)\setminus C$, $N=|C|$, and $L=\lceil A\log n\rceil$ for a
constant $A$ to be chosen. Sort the new pairs according as both endpoints
lie in $C$, both in $D$, or one on each side. Conditional on the three
counts, the lists are independent, and each list consists of uniform pairs
of its indicated kind. By binomial concentration there is $\eta>0$,
depending only on $a,b$, such that with high probability there are at
least $\eta n$ pairs within $C$ and, if $|D|\ge L$, at least $\eta|D|$
crossing pairs. The maximum number of new pairs incident to any one vertex
is $O(\log n)$ with high probability, by a binomial tail bound and a union
bound. We use this degree event only for estimates after exposure.

Write $J_C$ for $H[C]$ together with the new pairs within $C$.
Among $O(n)$ draws on $\binom{N}{2}$ pairs, the expected number of
colliding pairs of draws is $O(1)$. Thus with high probability at least
$\eta n/2$ distinct pairs are drawn. Conditional on their number, the
distinct pairs form a uniform set. In a uniform ordering of all pairs in
$C$, a prefix of this length contains a copy of $G(N,\delta/N)$ with high
probability. To see this, take an independent binomial prefix length with mean
$\delta(N-1)/2$ and choose $0<\delta<\eta/2$.
By the edge-expansion theorem of Krivelevich, Reichman and Samotij
\cite{KRS} for a connected graph perturbed by $G(N,\delta/N)$, we have
with high probability
\[
  \min_{0<|S|\le N/2}
  \frac{|E_{J_C}(S,C\setminus S)|}{|S|}\ge \frac{c_0}{\log n}
\]
for some $c_0>0$. Adding edges preserves this inequality.
The combinatorial Cheeger inequality bounds the first positive eigenvalue
of the unnormalized Laplacian below by the square of this expansion
divided by twice the maximum degree. Consequently, writing
$\bar f_C=N^{-1}\sum_{u\in C}f(u)$, we obtain
\begin{equation}\label{eq:core-poincare}
   \sum_{u\in C}(f(u)-\bar f_C)^2
       \le O((\log n)^3)\,\mathcal E_{J_C}(f).
\end{equation}

Expose the graph on $D$, leaving the crossing pairs unexposed.
Each component of size at least $L$ can be partitioned into connected
blocks of sizes between $L$ and $B=O((\log n)^2)$. In a rooted spanning
tree, cut off a deepest subtree of size at least $L$ and repeat. Each such
subtree has at most $1+\Delta(L-1)$ vertices, where $\Delta$ is the
maximum degree of the graph on $D$. Merge any final remainder of size
less than $L$ into an adjacent removed block. Keep each component of size
less than $L$ as one block. When $|D|\ge L$, the probability that a
given block of size at least $L$ receives no crossing pair is at most
$\exp(-\eta L)$, conditional on the exposed graph and the counts above.
Taking $A>3/\eta$, a union bound shows that every such block attaches to
$C$. A small block is an entire component on $D$, and belongs to $J$ if
and only if it receives a crossing pair. This also covers $|D|<L$.
All components outside $J$ therefore have fewer than $L$ vertices.

Assign every attached block to the endpoint $u\in C$ of one of its
crossing pairs. Each $u$, together with the blocks assigned to it, forms
a connected set $V_u$ of size at most $Q=1+O(B\log n)=O((\log n)^3)$.
The sets $V_u$ partition $V(J)$. Along a spanning tree of $V_u$ rooted
at $u$, Cauchy--Schwarz gives
\[
   \sum_{v\in V_u}(f(v)-f(u))^2
       \le Q^2\mathcal E_{J[V_u]}(f).
\]
Summing over $u$ and using \eqref{eq:core-poincare} yields
\[
 \sum_{v\in V(J)}(f(v)-\bar f_C)^2
 \le 2 Q^2\mathcal E_J(f)+2Q\sum_{u\in C}(f(u)-\bar f_C)^2
 \le O((\log n)^6)\mathcal E_J(f).
\]
This proves \eqref{eq:sparse-poincare}.
\end{proof}

\begin{proof}[Proof of Proposition~\ref{prop:upper}]
Expose the first $\lfloor qm\rfloor$ features, where $q$ is fixed by
\eqref{eq:trunc}. Their projection contains a component with at least
$an$ vertices with high probability, by the bounded-weight argument
in Section~\ref{sec:boundary}.

For the remaining features, reveal their sizes and the memberships of all
except those of size two. The number $M$ of exceptions is between two
positive constant multiples of $n$ with high probability. By bounded convergence the probability
of size two tends to
$\EE[e^{-W/\sqrt c}(W/\sqrt c)^2/2]>0$, and the sizes are independent.
Given this information, the $M$ unrevealed pairs are independent and
uniform. For each revealed feature connect its members in a path in label
order, ignoring empty features and singletons. The resulting simple graph
$H$ preserves connectivity within each revealed feature, so it contains
a component of size at least $an$. If $R_v$ is the total number of features
selected by $v$, then $\max_v R_v=O(\log n)$ with high probability.
Indeed, conditional on the weights, $R_v$ is a sum of independent
indicators with mean at most $\sum_j W_j/(n\sqrt c)=O(1)$ with high
probability. A Chernoff bound and a union bound suffice. Thus the
observable bound $\Delta(H)=O(\log n)$ holds with high probability
since $d_H(v)\le 2R_v$.

Apply Lemma~\ref{lem:sprinkle} conditionally to $H$ and the unrevealed
pairs. Restoring the clique edges preserves all components, hence
$V(J)=V(\cC)$ and $J$ is a spanning subgraph of $\cC$.
In the full projection $d(v)\le K_n^* R_v$, so
$\Delta(\cC)=O(K_n^*\log n)$ with high probability. The variational
formula for the lazy walk now gives
\[
  \trel
   =\sup_{f\,\mathrm{nonconstant}}
    \frac{2\min_{z\in\mathbb R}\sum_{v\in\cC}d(v)(f(v)-z)^2}
         {\mathcal E_{\cC}(f)}
   \le O(K_n^*(\log n)^7)
\]
by \eqref{eq:sparse-poincare} and
$\mathcal E_J(f)\le\mathcal E_{\cC}(f)$. Finally the least stationary
mass is at least $1/[n(n-1)]$. The standard relaxation bound for
total variation mixing \cite[Chapter 12]{LPW} supplies the final factor
of $O_\varepsilon(\log n)$.
\end{proof}

\section{Small eigenvalues}\label{sec:spectrum}

Let \(0=\lambda_1<\lambda_2\le\cdots\le\lambda_{|\cC|}\) be the
eigenvalues of \(I-P\) on the giant, counted with multiplicity, and set
\(N_n(h)=\#\{i:0<\lambda_i\le h\}\).

\begin{theorem}\label{thm:spectrum}
Suppose \(T=T_n\to\infty\) and \(M_n=m\bar F(\sqrt c\,T)\to\infty\).
For every \(\eta>0\),
\begin{equation}\label{eq:spectral-count}
  \PP\!\left(N_n\!\left(\frac{\nu}{2T}\right)
      \ge(1-\eta)M_n\right)\longrightarrow 1.
\end{equation}
At the extreme scale, for every fixed \(r\ge1\) and \(\eta>0\),
\begin{equation}\label{eq:ranked-modes}
  \PP\!\left(\lambda_{i+1}K_{J_i}\le(1+\eta)\frac{\nu}{2}
       \text{ for all }1\le i\le r\right)\longrightarrow1 .
\end{equation}
\end{theorem}

\begin{proof}
Fix \(1<u<v<\infty\). Mark features with
\(u\sqrt c\,T<W_j\le v\sqrt c\,T\). Their number divided by \(M_n\)
converges in probability to \(u^{-\beta}-v^{-\beta}\), by regular
variation and binomial concentration. Choose \(u'\in(1,u)\) and
\(\delta>0\) so that \((\nu+\delta)/u'<\nu\).
Conditional on any weight in the marked interval, binomial concentration
gives \(u'T\le K_j\le 2vT\) with probability tending uniformly to one.
Here \(T=o(\sqrt n)\), since \(x^2\bar F(x)\to0\) and
\(m\bar F(\sqrt c\,T)\to\infty\). Lemma~\ref{lem:boundary}
therefore gives \(b_j\le(\nu+\delta)K_j\) with the same probability,
where \(b_j\) is the number of edges leaving \(F_j\).

To check inclusion in the giant, fix a bounded-weight truncation with a
giant of at least \(\rho n\) vertices with high probability, for some
\(\rho>0\). A marked feature is outside this truncation for large \(n\).
Conditional on its weight, expose the truncated graph on the other
columns. The probability that its clique misses a truncated component
of size at least \(\rho n\), when one exists, is at most
\(\exp(-\rho uT)\). By summing the tagged failure probabilities and applying
Markov's inequality, we may discard \(o_{\mathbb P}(M_n)\)
marks so that the size and boundary estimates hold and each remaining
clique meets a linear component. On the high-probability event of
Proposition~\ref{prop:upper} all these cliques belong to \(\cC\).

Two marks conflict if their cliques intersect or an edge joins them.
Conditional on \(W_i,W_j\), let \(p_i,p_j\) denote their membership
probabilities, and let \(p\) be the membership probability of an
independent feature. A union bound over a common vertex, or over two
distinct vertices and a third feature witnessing an edge, bounds the
conflict probability by
\[
  np_ip_j+n(n-1)p_ip_j(m-2)\EE[p^2]
      =O_{v}(T^2/n).
\]
Indeed, \((m-2)\EE[p^2]=O(1/n)\). The expected number of conflicting
pairs of marks is \(O_v(M_n^2T^2/n)=o(M_n)\), since
\(T^2\bar F(\sqrt c\,T)\to0\). Discard the endpoints of these pairs.

For any remaining clique \(F_j\), its indicator has Rayleigh quotient
\[
  \frac{\langle\mathbf1_{F_j},(I-P)\mathbf1_{F_j}\rangle_\pi}
       {\langle\mathbf1_{F_j},\mathbf1_{F_j}\rangle_\pi}
  =\frac{b_j}{2\vol(F_j)}
  \le \frac{\nu+\delta}{2(K_j-1)}
  \le \frac{\nu}{2T}
\]
for all large \(n\). Here \(\vol(F_j)=K_j(K_j-1)+b_j\). The indicators
have disjoint supports with no edges between them, so their cross terms
vanish in both quadratic forms. If \(k\) indicators remain, the
min--max principle gives at least \(k-1\) positive eigenvalues at most
\(\nu/(2T)\). Taking \(u\) sufficiently close to one and \(v\)
sufficiently large proves \eqref{eq:spectral-count}.

For \eqref{eq:ranked-modes}, work first in a fixed band
\(\ell\sqrt c\,a_n<W_j\le L\sqrt c\,a_n\), with \(0<\ell<L<\infty\).
The band has bounded expected count. The tagged estimates and the same
pair bound (now \(O_L(a_n^2/n)=o(1)\)) imply that with high probability
all its cliques belong to the giant, are mutually nonadjacent and
disjoint, and satisfy \(K_j/(W_j/\sqrt c)=1+o_{\mathbb P}(1)\) and
\(b_j/K_j=\nu+o_{\mathbb P}(1)\), uniformly in the band.
The probability that the first \(r\) weights lie in the band can be
made arbitrarily close to one by taking \(\ell\) small and \(L\) large,
by the extreme-value limit for the weights. Fix \(i\le r\) and set
\(U=\bigcup_{h\le i} F_{J_h}\). Its stationary mass is
\(O_{\mathbb P}(a_n^2/n)=o_{\mathbb P}(1)\), since \(\vol(\cC)\) is
at least of order \(n\). For \(g=\sum_{h\le i} z_h\mathbf1_{F_{J_h}}\),
Cauchy--Schwarz gives \(\operatorname{Var}_\pi g\ge
(1-\pi(U))\langle g,g\rangle_\pi\). Apply min--max to the span of
these indicators and the constant function. This bounds \(\lambda_{i+1}\)
by \((1-\pi(U))^{-1}\max_{h\le i} b_{J_h}/(2\vol(F_{J_h}))\).
Since \(\min_{h\le i}K_{J_h}/K_{J_i}\to1\) in probability, the stated
bound follows simultaneously for \(i\le r\).
\end{proof}

\section{Pendant paths}\label{sec:paths}

Call a sequence of distinct vertices $v_0,\dots,v_k$ a pendant path of
length $k$ if consecutive vertices are adjacent, $v_k$ has degree one,
and $v_i$ has degree two for $1\le i<k$. The root $v_0$ may have any
degree. Let $L_n$ be the maximum such length in $\mathcal C_1$, with
$L_n=0$ if there is no such path. Put $\Lambda=W/\sqrt c$ and define
\[
  \xi=c\,\mathbb E[\Lambda^2 e^{-\Lambda}],\qquad
  d=c\,\mathbb E[\Lambda(1-e^{-\Lambda})],\qquad
  \theta=\xi e^{-d}.
\]
The parameter $\xi$ is the limiting mean number of incident size-two
features per vertex. The parameter $d$ is the limiting mean number of incident features of size at least two. Thus $0<\xi\le d$ and $0<\theta\le e^{-1}$.

\begin{theorem}\label{thm:pendant}
Under the assumptions above,
\[
            \frac{L_n}{\log n}\xrightarrow{\mathbb P}
                          \frac{1}{|\log\theta|}.
\]
\end{theorem}

\begin{proof}
Given
the vector of sizes, the features are independent uniform subsets of
their respective sizes. A feature is called active if $K_j\ge2$.
Write $M_2=\#\{j:K_j=2\}$ and $d_n=n^{-1}\sum_{j:K_j\ge2}K_j$.
Conditional on $W$, a feature size converges in law to
$\operatorname{Poisson}(\Lambda)$. Its size has uniformly bounded
second moment, since $\mathbb E W^2<\infty$. Independence across
features, the variance bound for their empirical means and the
mixed Poisson limit give
\[
     2M_2/n\xrightarrow{\mathbb P}\xi,\qquad
     d_n\xrightarrow{\mathbb P}d,\qquad
     \sum_jK_j^2=O_{\mathbb P}(n).
\]

For an integer $s$ let $a_b(s)=\binom{n-s}{b}/\binom nb$, the
probability that a uniform $b$-subset avoids $s$ specified vertices.
For any collection $J$ of active features put
$A_J(s)=\prod_{j\in J}a_{K_j}(s)$. Uniformly in $J$ and
$s=O(\log n)$,
\begin{equation}\label{eq:path-avoid}
 \log A_J(s)=-\frac{s}{n}\sum_{j\in J}K_j+o_{\mathbb P}(1).
\end{equation}
Indeed, $a_b(s)=\prod_{t=0}^{s-1}(1-b/(n-t))$. On the event
$\max_jK_j\le n/3$, which has probability tending to one, the
summed error in its logarithmic expansion is
$O((s^2\sum_jK_j+s\sum_jK_j^2)/n^2)=o_{\mathbb P}(1)$.

For the upper bound we can count paths in the entire graph. Each
active feature meeting $v_1,\dots,v_k$ must consist of the endpoints
of one path edge. Indeed, a larger feature would form a triangle through
one of the prescribed vertices, contradicting the degree requirements
at the leaf or at a vertex on its way to the leaf. Assign a distinct
size-two feature to each path edge. Other size-two features can either
avoid the $k$ prescribed vertices or duplicate one of these $k$ edges.
Write $N=\binom n2$, $(z)_r=z(z-1)\cdots(z-r+1)$, and let $J$ contain
all active features. The conditional expected number of paths is at most
\[
 (n)_{k+1}(M_2)_k N^{-k}\,
 \frac{A_J(k)}{a_2(k)^k}
 \left(1+\frac{k}{N a_2(k)}\right)^{M_2-k}.
\]
The last factor is $1+O(k/n)$, and $a_2(k)^{-k}=1+O(k^2/n)$.
By \eqref{eq:path-avoid} the bound is at most
$(1+o_{\mathbb P}(1))n[(2M_2/n)e^{-d_n}]^k$ for $k=O(\log n)$.
For $k=\lceil(1+\varepsilon)\log n/|\log\theta|\rceil$ this tends to
zero in probability. Conditional Markov's inequality gives the upper
bound, since the leaf end of any longer path also gives a path of
length $k$.

For the lower bound, put $b_n=\lceil(\log n)^2\rceil$. Expose the
members of all features with sizes in $[b_n,2b_n]$ and write $R$ for
their union. With probability tending to one all these features lie
in $\mathcal C_1$, by Proposition~\ref{prop:upper}. We also have
\begin{equation}\label{eq:path-root-set}
       \log |R|/\log n\xrightarrow{\mathbb P}1,
       \qquad |R|/n\xrightarrow{\mathbb P}0.
\end{equation}
To see this, the number of weights in
$[5\sqrt c\,b_n/4,7\sqrt c\,b_n/4]$ is $n^{1-o(1)}$ with probability
tending to one, by regular variation and binomial concentration.
Every corresponding feature has size in $[b_n,2b_n]$ with probability
tending to one, by a union bound and a conditional binomial tail bound
of order $e^{-\mathrm{const}\, b_n}$. The sum $T$ of the sizes in
the exposed band is at most $b_n^{-1}\sum_jK_j^2=o_{\mathbb P}(n)$.
Conditional on the sizes, the expected number of repeated memberships
$T-|R|$ is at most $T^2/(2n)$, by counting pairwise intersections.
This proves \eqref{eq:path-root-set}.

Condition on the sizes and exposed memberships. Count ordered paths
with $v_0\in R$, all other vertices outside $R$, a distinct assigned
size-two feature on each edge, and every other unexposed active feature
avoiding all $k+1$ vertices. Let $X$ be this count and $J'$ the
unexposed active features. Its conditional mean is exactly
\[
  \mu=|R|(n-|R|)_k(M_2)_k N^{-k}\,
                  \frac{A_{J'}(k+1)}{a_2(k+1)^k}.
\]
Since $n^{-1}\sum_{j\in J'}K_j=d_n-T/n\to d$ in probability,
\eqref{eq:path-avoid} and \eqref{eq:path-root-set} give
$\log\mu=\log n+k\log\theta+o_{\mathbb P}(\log n)$.
For $0<\varepsilon<1$, choosing $k=\lfloor(1-\varepsilon)\log n/|\log\theta|\rfloor$ thus gives $\mu\to\infty$ in probability.

Two distinct candidates that can occur together have disjoint vertex
sets and disjoint assigned features. Indeed, each root has exactly one
unexposed active feature, and an intersection in the interior forces
the same chain in both directions up to its root in $R$ and its leaf.
For two disjoint candidates every unassigned feature of size $b$ gives
a joint avoidance factor $a_b(2s)\le a_b(s)^2$, where $s=k+1$.
Each of the $2k$ assigned features automatically avoids the other
candidate's vertices. The joint probability is therefore at most
$a_2(s)^{-2k}=1+O(k^2/n)$ times the product of the individual
probabilities. It follows that
$\operatorname{Var}(X\mid\text{exposure})\le\mu+O(k^2/n)\mu^2$.
Chebyshev's inequality now gives $X>0$ with probability tending to
one. As $R\subseteq\mathcal C_1$ with probability tending to one, this
proves the lower bound.
\end{proof}

For a pendant path of length $k$, start the lazy walk at its leaf and let $T_0$ be its first visit to the root. Up to $T_0$, its distance from the leaf has the law of the absolute value of a half-lazy symmetric walk on $\mathbb Z$ stopped on reaching distance $k$. In particular $\EE_{v_k}T_0=2k^2$, by the elementary gambler's ruin equations. Theorem~\ref{thm:pendant} thus gives the constant $2/|\log\theta|^2$ in front of $(\log n)^2$ for the mean traversal time of a longest pendant path. By Theorem~\ref{thm:main}, this delay is of smaller order than worst-case mixing.

\end{document}